\documentclass[12pt,a4paper,final]{article}

\usepackage{amsfonts}
\usepackage{amsmath}
\usepackage{amssymb}
\usepackage{graphicx}

\newtheorem{theorem}{Theorem}

\newtheorem{corollary}[theorem]{Corollary}

\newtheorem{definition}[theorem]{Definition}

\newtheorem{lemma}[theorem]{Lemma}

\newtheorem{proposition}[theorem]{Proposition}
\newtheorem{remark}[theorem]{Remark}

\newenvironment{proof}[1][Proof]{\noindent\textbf{#1.} }{\ \rule{0.5em}{0.5em}}

\begin{document}

\title{Lower bounds for transition probabilities on graphs}
\author{Andr\'{a}s Telcs \\
{\small Department of Computer Science and Information Theory, }\\
Budapest {\small University of Technology and Economics}\\
{\small telcs@szit.bme.hu}}
\maketitle

\begin{abstract}
The paper presents two results. \ The first one provides separate conditions
for the upper and lower estimate of the distribution of the exit time from
balls of a random walk on a weighted graph. The main result of the paper is
that the lower estimate follows from the elliptic Harnack inequality. \ The
second result is an off-diagonal lower bound for the transition probability
of the random walk.
\end{abstract}

\section{Introduction}

Today a large amount of work is devoted to upper and two-sided estimates of
heat kernels in different spaces (c.f. \cite{BCG},\cite{C},\cite{CG},\cite%
{HR},\cite{Ra1}).\ \ The main challenge is to find a connection between
structural properties of the space and the behavior of the heat kernel. The
study of the heat kernel in $%
\mathbb{R}
^{n}$ of course dates back to much earlier results among others to Moser
\cite{M1},\cite{M2} and Aronson \cite{A}. In these celebrated works chaining
arguments were used. \ Chaining arguments appear in recent works as well. \
The present paper would like to provide a new one which replaces Aronson's
chaining argument for graphs to obtain heat kernel lower estimates. \ The
new approach eliminates the condition on the volume growth.

It is generally believed that the majority of the essential phenomena and
difficulties related to diffusion are present in the discrete case. All that
follows is in the discrete graph settings and discrete time, but one can see
that most of the arguments carry over to the continuous case.

In the course of the study of the pre-Sierpinski gasket ( c.f.\cite{J} \cite%
{B1} and bibliography there) and other fractal structures upper
or two-sided heat kernel estimates were given, which in the
simplest case has the form as follows:
\begin{equation}
p_{n}(x,y)+p_{n+1}(x,y)\geq \frac{c}{V(x,n^{1/\beta })}\exp \left( -C\left(
\frac{d^{\beta }(x,y)}{n}\right) ^{\frac{1}{\beta -1}}\right)  \label{LEb}
\end{equation}%
\begin{equation}
p_{n}(x,y)\leq \frac{C}{V(x,n^{1/\beta })}\exp \left( -c\left( \frac{%
d^{\beta }(x,y)}{n}\right) ^{\frac{1}{\beta -1}}\right)  \label{UEb}
\end{equation}%
In \cite{GT2} necessary an sufficient condition were given for $\left( \ref%
{LEb}\right) $ and $\left( \ref{UEb}\right) $. The standard route to the
lower estimate typically goes via the diagonal upper and lower bound (and
uses $\left( \ref{PUE}\right) $). \ The present paper develops a different
approach, which uses fewer assumptions. \ Neither volume growth conditions
nor heat kernel upper estimates are used. Let us mention here that in \cite%
{BCK} such estimates are given for strongly recurrent graphs without
explicitly assuming the elliptic Harnack inequality. Meanwhile it is easy to
show that the elliptic Harnack inequality follows directly from the
conditions there.

During the proof of the upper estimate an interesting side-result can be
observed.\ The distribution of the exit time from a ball has an upper
estimate under a particular condition. Consider $T_{B}$, the exit time from
a ball $B=B\left( x,R\right) .$ The expected value of $T_{B}$ is denoted by $%
E\left( x,R\right) =E\left( T_{B}|X_{0}=x\right) $ assuming that the
starting point is $x.$ On many fractals (or fractal type graph) the
space-time scaling function is $R^{\beta },$ $cR^{\beta }\leq E\left(
x,R\right) \leq CR^{\beta },$ for $\beta \geq 2,C>1>c>0$ constants and this
property implies that
\begin{equation}
\mathbb{P}\left( T_{B}<n|X_{0}=x\right) \leq C\exp \left( -c\left( \frac{%
R^{\beta }}{n}\right) ^{\frac{1}{\beta -1}}\right) .  \label{PUE}
\end{equation}%
This estimate (and the lower counterpart as well in the case of the Brownian
motion on the Sierpinski gasket) was given first in \cite{BP} and later an
independent proof was provided for more general settings in \cite{GT1} using
also a chaining argument.

One might wonder about the condition which ensures the same (up to the
constants) lower bound.

The main results are illustrated for the particular case $cR^{\beta }\leq
E\left( x,R\right) \leq CR^{\beta }$ postponing the general statements after
the necessary definitions. If the elliptic Harnack inequality (see
Definition $\left( \ref{EHI}\right) $) holds, then for $n\geq R,B=B\left(
x,R\right) $%
\begin{equation}
\mathbb{P}\left( T_{B}<n|X_{0}=x\right) \geq c\exp \left( -C\left( \frac{%
R^{\beta }}{n}\right) ^{\frac{1}{\beta -1}}\right) .  \label{eetdle}
\end{equation}%
and
\begin{equation}
p_{n}(x,y)+p_{n+1}(x,y)\geq \frac{c}{V(x,n^{1/\beta })r^{D}}\exp \left(
-C\left( \frac{d^{\beta }(x,y)}{n}\right) ^{\frac{1}{\beta -1}}\right) ,
\label{LEbb}
\end{equation}

where $r=\left( \frac{n}{d\left( x,y\right) }\right) ^{\frac{1}{\beta -1}%
},n\geq d\left( x,y\right) \geq 0,n>0$ \textbf{\ }and\textbf{\ }$D$ is a
fixed constant.

The results are new from several points of view. \ First of all, to our best
knowledge, lower estimates like $\left( \ref{eetdle}\right) $ are new in
this generality. \ One should also observe that the lower estimate $\left( %
\ref{eetdle}\right) $ matches with the upper one $\left( \ref{PUE}\right) $
obtained from stronger assumptions. The key steps are given in Proposition %
\ref{p2} and \ref{p3} which help to control the probability to hit a nearby
ball, which is usually more difficult than to control exit from a ball.

In Section \ref{sdef} the necessary definitions are introduced. \ In Section %
\ref{stdist} we give the general form and proof of $\left( \ref{eetdle}%
\right) $. In Section \ref{sVSR} we show a heat kernel lower bound (better
than $\left( \ref{LEbb}\right) $) for very strongly recurrent walks and in
Section \ref{sle} we show a result which contains $\left( \ref{LEbb}\right) $
as a particular case.

\textbf{Acknowledgement}

The author is indebted to Professor Alexander Grigor'yan for many useful
discussions and particularly for the remarks helped to clarify the proof of
Lemma \ref{lptt>c}. Thanks are due to the referee for the careful reading of
the paper and many helpful suggestions.

\section{Basic definitions}

\label{sdef}In this section we give the basic definitions for our
discussion. \ Let us consider an infinite connected graph $\Gamma $. \ We
assume, for sake of simplicity, that there are no multiple edges and loops.

Let $\mu _{x,y}=\mu _{y,x}>0$ be a symmetric weight function given on the
edges $x\sim y.$ These weights induce a measure $\mu (x)$%
\begin{eqnarray*}
\mu (x) &=&\sum_{y\sim x}\mu _{x,y}, \\
\mu (A) &=&\sum_{y\in A}\mu (y)
\end{eqnarray*}%
on the vertex sets $A\subset \Gamma .$ \ The weights $\mu _{x,y}$ define a
reversible Markov chain $X_{n}\in \Gamma $, i.e., a random walk on the
weighted graph $(\Gamma ,\mu )$ with transition probabilities
\begin{align*}
P(x,y)& =\frac{\mu _{x,y}}{\mu (x)}, \\
P_{n}(x,y)& =\mathbb{P}(X_{n}=y|X_{0}=x).
\end{align*}%
The transition \textquotedblright density\textquotedblright\ or heat kernel
for the discrete random walk is defined as%
\begin{equation*}
p_{n}\left( x,y\right) =\frac{1}{\mu \left( y\right) }P_{n}\left( x,y\right)
.
\end{equation*}%
To avoid parity problems we introduce
\begin{equation*}
\widetilde{p}_{n}\left( x,y\right) =p_{n}\left( x,y\right) +p_{n+1}\left(
x,y\right) .
\end{equation*}%
We will assume in the whole paper that the one step transition probabilities
are uniformly separated from zero, i.e. there is a $p_{0}>0$ such that
\begin{equation}
P\left( x,y\right) \geq p_{0}>0  \tag{$p_{0}$}  \label{p0}
\end{equation}%
for all $x\sim y,$ $\ \,x,y\in \Gamma .$

\begin{definition}
The graph is equipped with the usual (shortest path length) graph distance $%
d(x,y)$ and open metric balls are defined for $x\in \Gamma ,$ $R>0$ as
\begin{eqnarray*}
B(x,R) &=&\{y\in \Gamma :d(x,y)<R\}, \\
S\left( x,R\right) &=&\{y\in \Gamma :d(x,y)=R\}
\end{eqnarray*}%
and the $\mu -$measure of $B\left( x,R\right) $ denoted by $V(x,R)$%
\begin{equation*}
V\left( x,R\right) =\mu \left( B\left( x,R\right) \right) .
\end{equation*}
\end{definition}

\begin{definition}
We use
\begin{equation*}
\overline{A}=\left\{ y\in \Gamma :\exists x\in A,x\sim y\right\}
\end{equation*}%
for the closure of a set $A,$ Denote $\partial A=\overline{A}\backslash A\ $%
and $A^{c}=\Gamma \backslash A$ the complement of $A.$\bigskip
\end{definition}

\begin{definition}
In general, $a_{\xi }\simeq b_{\xi }$ will mean that there is a $C>0$ such
that for all $\xi $%
\begin{equation*}
\frac{1}{C}a_{\xi }\leq b_{\xi }\leq Ca_{\xi }.
\end{equation*}
\end{definition}

Unimportant constants will be denoted by $c,C$ and they may change from
place to place absorbing other intermediate constants.

Let us introduce the exit time $T_{A}$ \ for a set $A\subset \Gamma .$

\begin{definition}
The exit time from a set $A$ is defined as
\begin{equation*}
T_{A}=\inf \{t\geq 0:X_{n}\in A^{c}\},
\end{equation*}%
its expected value is denoted by
\begin{eqnarray*}
E_{y}(A) &=&\mathbb{E}(T_{A}|X_{0}=y), \\
E_{y}\left( x,R\right) &=&E_{y}\left( B\left( x,R\right) \right)
\end{eqnarray*}%
and we will use the $E=E(x,R)=E_{x}\left( B\left( x,R\right) \right) $ and $%
T_{x,R}=T_{B\left( x,R\right) }$ short notations.
\end{definition}

The definition implies that
\begin{equation}
E\left( x,1\right) =1.  \label{one}
\end{equation}

\begin{definition}
The hitting time $\tau _{A}$ of a set $A\subset \Gamma $ is defined by
\begin{equation*}
\tau _{A}=T_{A^{c}},
\end{equation*}%
and we write $\tau _{x,R}=\tau _{B\left( x,R\right) }$.
\end{definition}

\begin{definition}
We introduce the maximal exit time for $x\in \Gamma ,R>0$ by
\begin{equation*}
\overline{E}\left( x,R\right) =\max\limits_{y\in B\left( x,R\right)
}E_{y}\left( x,R\right) .
\end{equation*}
\end{definition}

\begin{definition}
One of the key assumptions in our study is the condition $\left( \overline{E}%
\right) $: there is a $C>0$ such that for all $x\in \Gamma ,R>0$%
\begin{equation}
\overline{E}\left( x,R\right) \leq CE\left( x,R\right)
\end{equation}%
is true.
\end{definition}

\begin{definition}
We say that the \emph{time comparison principle} holds for $\left( \Gamma
,\mu \right) $ if there is a $C_{T}>1$ constant such that for any $x\in
\Gamma ,R>0,y\in B\left( x,R\right) $

\begin{equation}
\frac{E\left( y,2R\right) }{E\left( x,R\right) }\leq C_{T}.  \label{TC}
\end{equation}
\end{definition}

\begin{proposition}
From the time comparison principle it follows that
\begin{equation}
\frac{E\left( x,2R\right) }{E\left( x,R\right) }\leq C_{T},  \label{TD}
\end{equation}%
\begin{equation}
\overline{E}\left( x,R\right) \leq CE\left( x,R\right)  \label{Ebar}
\end{equation}%
and there is a constant $A_{T}$ such that for all $x\in \Gamma ,R>0$
\begin{equation}
E\left( x,A_{T}R\right) \geq 2E\left( x,R\right) .  \label{PD3E}
\end{equation}
\end{proposition}

\begin{remark}
For the easy proofs see \cite{Td}. One can deduce that $\left( \ref{TD}%
\right) $ is equivalent to that there is a $\beta \geq 1$ and $C>0$ such
that for all $R>r>0,x\in \Gamma ,y\in B\left( x,R\right) $
\begin{equation}
\frac{E\left( x,2R\right) }{E\left( x,r\right) }\leq C\left( \frac{2R}{r}%
\right) ^{\beta },  \label{beta}
\end{equation}%
and it implies
\begin{equation*}
E\left( x,R\right) \leq CR^{\beta }.
\end{equation*}%
Similarly $\left( \ref{PD3E}\right) $ is equivalent to that there are $\beta
^{\prime }>0,c>0$ such that for all $x\in \Gamma ,R>r>0,y\in B\left(
x,R\right) $
\begin{equation}
c\left( \frac{2R}{r}\right) ^{\beta ^{\prime }}\leq \frac{E\left(
x,2R\right) }{E\left( x,r\right) }  \label{betaprime}
\end{equation}%
and from $\left( \ref{beta}\right) $ it follows that%
\begin{equation*}
E\left( x,R\right) \geq cR^{\beta ^{\prime }}.
\end{equation*}
\end{remark}

\begin{remark}
\label{rEbar}It is also easy to see that $\left( \overline{E}\right) $
implies $\left( \ref{PD3E}\right) $ and hence $\left( \ref{betaprime}\right)
$ as well.
\end{remark}

\begin{definition}
For the mean exit time $E\left( x,R\right) ,$ $R\in \mathbb{N}$ \ we define
the inverse in \ the second variable
\begin{equation*}
e\left( x,n\right) =\min \left\{ r\in \mathbb{N}:E\left( x,r\right) \geq
n\right\} .
\end{equation*}
\end{definition}

\begin{remark}
The inverse function $e\left( x,n\right) $ is well-defined since $E\left(
x,R\right) $ is strictly increasing for $R\in \mathbb{N}$ (cf. \cite{TER}).
\end{remark}

\begin{definition}
For a given $x\in \Gamma ,n\geq R>0$ let us define $k=k\left( x,n,R\right) $
as the maximal integer for which
\begin{equation*}
\frac{n}{k}\leq q\min\limits_{z\in B\left( x,R\right) }E\left( z,\frac{R}{k}%
\right) ,
\end{equation*}%
where $q$ is a fixed constant. \ Let $k=1$ by definition if there is no such
integer.
\end{definition}

\begin{definition}
Let us denote by $\pi _{x,y}$ the the union of the vertices of shortest
paths connecting $x$ \ and $y$.
\end{definition}

\begin{definition}
For $x,y\in \Gamma ,n\geq R>0,C>0$ let us define \newline
$l=l_{C}\left( x,y,n,R\right) $ as the minimal integer for which
\begin{equation*}
\frac{n}{l}\geq Q\max\limits_{z\in \pi _{x,y}}E\left( z,\frac{CR}{l}\right) ,
\end{equation*}%
where $Q$ is a fixed constant (to \ be specified later.), \ Let $l=R$ by
definition if there is no such integer. \ If $d\left( x,y\right) =R$ we will
use the shorter notation $l_{C}\left( x,y,n\right) =l_{C}\left(
x,y,n,d\left( x,y\right) \right) $
\end{definition}

\begin{definition}
For a given $x\in \Gamma ,n\geq R>0$ let us define
\begin{equation*}
\nu =\nu \left( x,n,R\right) =\min\limits_{y\in S\left( x,2R\right)
}l_{9}\left( x,y,n,R\right) .
\end{equation*}
\end{definition}

\begin{remark}
One can show easily from $\left( \ref{beta}\right) $ that%
\begin{equation*}
k\left( x,n,R\right) \geq c\left( \frac{E\left( x,R\right) }{n}\right) ^{%
\frac{1}{\beta -1}}
\end{equation*}%
and similarly using $\left( \ref{betaprime}\right) $ that if $\beta ^{\prime
}>1$ that%
\begin{equation*}
\nu \left( x,n,R\right) \leq C\left( \frac{E\left( x,R\right) }{n}\right) ^{%
\frac{1}{\beta ^{\prime }-1}}.
\end{equation*}
\end{remark}

\begin{definition}
A function $h:\Gamma \rightarrow \mathbb{R}$ said to be \textit{harmonic} on
$A\subset \Gamma $ \ if it is defined on $\overline{A}$ and
\begin{equation*}
\sum_{y\in \Gamma }P\left( x,y\right) h\left( y\right) =h\left( x\right)
\text{ \ for all }x\in A.
\end{equation*}
\end{definition}

\begin{definition}
\label{EHI}We say that the weighted graph $(\Gamma ,\mu )$ satisfies $\left(
H\right) $ \emph{the elliptic Harnack inequality } if there is a constant $%
C>0$ such that for all $x\in \Gamma ,R>0$ and for any non-negative harmonic
function $u$ which is harmonic on $B(x,2R)$, the following inequality holds
\begin{equation*}
\max_{B\left( x,R\right) }u\leq C\min_{B\left( x,R\right) }u\,.
\end{equation*}
\end{definition}

\bigskip If the weights of the edges are\ considered as wires, the whole
graph can be seen as an electric network. \ Resistances are defined using
the usual capacity notion.

\begin{definition}
On $\left( \Gamma ,\mu \right) $ the Dirichlet form is defined as
\begin{equation*}
\mathcal{E}\left( f,f\right) =\sum_{y\sim z}\mu _{y,z}\left( f\left(
y\right) -f\left( z\right) \right) ^{2}
\end{equation*}
and the inner product is
\begin{equation*}
\left( f,f\right) =\sum_{y}f^{2}\left( x\right) \mu \left( x\right) .
\end{equation*}
\end{definition}

\begin{definition}
For any disjoint sets $A,B$ the capacity is defined via the Dirichlet form $%
\mathcal{E}$ by
\begin{equation*}
cap\left( A,B\right) =\inf \left\{ \mathcal{E}\left( f,f\right) :\
f|_{A}=1,f|_{B}=0\right\} .
\end{equation*}%
The resistance is defined then as
\begin{equation*}
\rho \left( A,B\right) =\frac{1}{cap\left( A,B\right) }.
\end{equation*}%
In particular we will use the following notations: for $R>r>0,x\in \Gamma $
\begin{equation*}
\rho \left( x,r,R\right) =\rho \left( B\left( x,r\right) ,B^{c}\left(
x,R\right) \right) .
\end{equation*}
\end{definition}

\section{Distribution of the exit time}

\label{stdist}In this section we prove the following theorem.

\begin{theorem}
\label{tptf}Assume that the weighted graph $\left( \Gamma ,\mu \right) $
satisfies $\left( p_{0}\right) $. \newline
1. If $\left( \overline{E}\right) $ holds , then there are \ $c,C>0$ such
that for all $n\geq R>0,x\in \Gamma $
\begin{equation*}
\mathbb{P}\left( T_{x,R}<n\right) \leq C\exp \left( -ck\left( x,n,R\right)
\right)
\end{equation*}%
is true.\newline
\newline
2. If $\left( \Gamma ,\mu \right) $ satisfies the elliptic Harnack
inequality $\left( H\right) $, then there are $c,C>0$ such that for all $%
n\geq R>0,x\in \Gamma $%
\begin{equation}
\mathbb{P}\left( T_{x,R}<n\right) \geq c\exp \left( -C\nu \left(
x,n,R\right) \right) .  \label{lb}
\end{equation}
\end{theorem}

The proof of the upper bound was given in \cite{Td}. The lower bound is
based on a new chaining argument. \ First we need some propositions.

\begin{proposition}
\label{p1}Assume that the weighted graph $\left( \Gamma ,\mu \right) $
satisfies $\left( p_{0}\right) $ and $\left( \overline{E}\right) $, then
there is a $c>0$ such that for all $x\in \Gamma ,n,R>0$%
\begin{equation*}
\mathbb{P}\left( T_{x,R}>n\right) >c,
\end{equation*}%
if $n\leq \frac{1}{4}E\left( x,R\right) $.
\end{proposition}

\begin{proof}
From Lemma 5.3 of \cite{Td} one has for $A=B\left( x,R\right) $ that%
\begin{equation*}
\mathbb{P}\left( T_{x,R}\leq n\right) \leq 1-\frac{E\left( x,R\right) }{2%
\overline{E}\left( x,R\right) }+\frac{n}{\overline{E}\left( x,R\right) }.
\end{equation*}%
From the condition $\frac{\overline{E}\left( x,R\right) }{E\left( x,R\right)
}\leq C$ and $n\leq \frac{1}{4}E\left( x,R\right) $ one obtains%
\begin{equation*}
\mathbb{P}\left( T_{x,R}>n\right) \geq \frac{E\left( x,R\right) -2n}{2%
\overline{E}\left( x,R\right) }\geq \frac{1}{4C}.
\end{equation*}
\end{proof}

\begin{lemma}
\label{lhg}If $\left( \Gamma ,\mu \right) $ satisfies $\left( p_{0}\right) $
and the elliptic Harnack inequality $\left( H\right) $, then\ for $x\in
\Gamma ,r>0,K>L\geq 1,B=B\left( x,Kr\right) ,S=\left\{ y:d\left( x,y\right)
=Lr\right\} $%
\begin{equation}
\min\limits_{w\in S}g^{B}\left( w,x\right) \simeq \rho \left( x,Lr,Kr\right)
\simeq \max\limits_{v\in S}g^{B}\left( v,x\right) .  \label{ehg}
\end{equation}
\end{lemma}

\begin{proof}
See Barlow's proof (\cite{Bnew}, Proposition 2).
\end{proof}

\begin{lemma}
\label{lptt>c}If $\left( \Gamma ,\mu \right) $ satisfies $\left(
p_{0}\right) $ and the elliptic Harnack inequality $\left( H\right) $, then
there is a $c_{1}>0$ such that for all $x\in \Gamma ,r>0,w\in \overline{B}%
\left( x,4r\right) $%
\begin{equation}
\mathbb{P}_{w}\left( \tau _{x,r}<T_{x,5r}\right) >c_{1}.  \label{tt>c}
\end{equation}
\end{lemma}

\begin{proof}
The investigated probability
\begin{equation}
u\left( w\right) =\mathbb{P}_{w}\left( \tau _{x,r}<T_{x,5r}\right)
\end{equation}%
is the capacity potential between $\Gamma \backslash B\left( x,5r\right) $
and $B\left( x,r\right) $ and clearly harmonic in $A=B\left( x,5r\right)
\backslash B\left( x,r\right) $. \ Write $B=B\left( x,5r\right) .$ So it can
be as usual decomposed
\begin{equation*}
u\left( w\right) =\sum_{z}g^{B}\left( w,z\right) \pi \left( z\right)
\end{equation*}%
with the proper capacity measure $\pi \left( z\right) $ with support in $%
S\left( x,r\right) $, $\pi \left( A\right) =1/\rho \left( x,r,5r\right) $. \
From the maximum (minimum) principle it follows that the minimum of $u\left(
w\right) $ is attained on the boundary, $w\in S\left( x,4r-1\right) $ and
from the Harnack inequality for $g^{B}\left( w,.\right) $ in $B\left(
x,2r\right) $ that
\begin{equation*}
\min_{z\in \overline{B}\left( x,r+1\right) }g^{B}\left( w,z\right) \geq
cg^{B}\left( w,x\right) ,
\end{equation*}%
\begin{equation*}
u\left( w\right) =\sum_{z}g^{B}\left( w,z\right) \pi \left( z\right) \geq
\frac{cg^{B}\left( w,x\right) }{\rho \left( x,r,5r\right) }.
\end{equation*}%
From Lemma \ref{lhg} we know that
\begin{equation*}
\max_{y\in B\left( x,5r\right) \backslash B\left( x,4r\right) }g^{B}\left(
y,x\right) \simeq \rho \left( x,4r,5r\right) \simeq \min_{w\in B\left(
x,4r\right) }g^{B}\left( w,x\right) .
\end{equation*}%
which means that
\begin{equation}
u\left( w\right) \geq c\frac{\rho \left( x,4r,5r\right) }{\rho \left(
x,r,5r\right) }.  \label{ubig}
\end{equation}%
Similarly from Lemma \ref{lhg} it follows that
\begin{equation*}
\max_{v\in B\left( x,5r\right) \backslash B\left( x,r\right) }g^{B}\left(
v,x\right) \simeq \rho \left( x,r,5r\right) \simeq \min_{w\in B\left(
x,r\right) }g^{B}\left( w,x\right) .
\end{equation*}%
Finally if $y_{0}\in \partial B\left( x,r\right) $ is on the ray from $x$ to
$y\in \partial B\left( x,4r\right) $ then iterating the Harnack inequality
along a finite chain of balls of radius $r/4$ along this ray from $y_{0}$ to
$y$ one obtains%
\begin{equation*}
g^{B}\left( y,x\right) \simeq g^{B}\left( y_{0},x\right) ,
\end{equation*}%
which results that%
\begin{equation*}
\rho \left( x,4r,5r\right) \geq c\rho \left( x,r,5r\right) ,
\end{equation*}%
and the statement follows from $\left( \ref{ubig}\right) $.
\end{proof}

\begin{proposition}
\label{p2}Assume that the weighted graph $\left( \Gamma ,\mu \right) $
satisfies $\left( p_{0}\right) $ and $\left( H\right) $. Then there are $%
c_{0},c_{1}>0$ such that for all $x,z\in \Gamma ,r>0,d\left( x,z\right) \leq
4r,m>\frac{2}{c_{1}}E\left( x,9r\right) $
\begin{equation*}
\mathbb{P}_{x}\left( \tau _{z,r}<m\right) >c_{0}.
\end{equation*}
\end{proposition}

\begin{proof}
We start with the following simple estimate:%
\begin{eqnarray*}
\mathbb{P}_{x}\left( \tau _{z,r}<m\right) &\geq &\mathbb{P}_{x}\left( \tau
_{z,r}<T_{x,9r}<m\right) \\
&=&\mathbb{P}_{x}\left( \tau _{z,r}<T_{x,9r}\right) -\mathbb{P}_{x}\left(
\tau _{z,r}<T_{x,9r},T_{x,9r}\geq m\right) \\
&\geq &\mathbb{P}_{x}\left( \tau _{z,r}<T_{x,9r}\right) -\mathbb{P}%
_{x}\left( T_{x,9r}\geq m\right) .
\end{eqnarray*}%
On one hand
\begin{equation*}
\mathbb{P}_{x}\left( T_{x,9r}\geq m\right) \leq \frac{E\left( x,9r\right) }{m%
}\leq \frac{E\left( x,9r\right) }{\frac{2}{c_{1}}E\left( x,9r\right) }%
<c_{1}/2
\end{equation*}%
and on the other hand $B\left( z,5r\right) \subset B\left( x,9r\right) ,$
hence
\begin{equation*}
\mathbb{P}_{x}\left( \tau _{z,r}<T_{x,9r}\right) \geq \mathbb{P}_{x}\left(
\tau _{z,r}<T_{z,5r}\right) ,
\end{equation*}%
and Lemma \ref{lptt>c} can be applied to get
\begin{equation*}
\mathbb{P}_{x}\left( \tau _{z,r}<T_{z,5r}\right) \geq c_{1}.
\end{equation*}%
The result follows with $c_{0}=c_{1}/2.$
\end{proof}

\begin{lemma}
\label{lchain}Let us assume that $x\in \Gamma ,m,r,l\geq 1,$ $0\leq u\leq
3l-2,$ $r=\left( 3l-2\right) r-u,y\in S\left( x,r+r\right) $ and write $n=ml$%
, then
\begin{equation*}
\mathbb{P}_{x}\left( \tau _{y,r}<n\right) \geq \min\limits_{w\in \pi
_{x,y},2r-3\leq d\left( z,w\right) \leq 4r}\mathbb{P}_{z}^{l}\left( \tau
_{w,r}<m\right) .
\end{equation*}%
where $\pi _{x,y}$ is the union of vertices of all possible shortest paths
from $x$ to $y$.
\end{lemma}

\begin{remark}
The statement (and its consequences) can be sharpened if we consider
separately all possible paths of comparable length to the shortest one and
consider the minimum over the vertices of each path than the maximum for the
paths. \ We omit this refinement here.
\end{remark}

\begin{proof}
We define a chain of balls . \ For $1\leq l\leq d\left(
x,y\right) -r$ let us consider a sequence of vertices $%
x_{0}=x,x_{1},...x_{l}=y,x_{i}\in \pi _{x,y}$ in the following way: $d\left(
x_{i-1},x_{i}\right) =r-\delta _{i},$ where $\delta _{i}\in \left\{
0,1,2,3\right\} $ for $i=1...l$ and
\begin{equation*}
u=\sum_{i=1}^{l}\delta _{i}
\end{equation*}%
\begin{equation*}
R=\left( 3l-2\right) r-\sum_{i=1}^{l}\delta _{i}=\left( 3l-2\right) r-u.
\end{equation*}%

Let $\tau _{i}=\tau _{x_{i},r}$ and $%
s_{i}=\tau _{i}-\tau _{i-1},A_{i}=\left\{ s_{i}<m\right\} ,\mathbb{A}%
_{i}=\cap _{j=1}^{i}A_{j}$ for $i=1,...l,\tau _{0}=0.$ Let us use the
notation $D_{i}\left( z_{i}\right) =A_{i}\cap \left\{ X_{\tau
_{i}}=z_{i}\right\} .$ \ One can observe that $\cap _{i=1}^{l}A_{i}$ means
that the walk takes less than $m$ steps between the first hit of the
consecutive $B_{i}=B\left( x_{i},r\right) $ balls, consequently
\begin{equation*}
\mathbb{P}_{x}\left( \tau _{y,r}<n\right) \geq \mathbb{P}_{x}\left( \mathbb{A%
}_{l}\right)
\end{equation*}%
We also note that $s_{i}=\min \left\{ k:X_{k}\in B_{i}|X_{0}\in \partial
B_{i-1}\right\} $. From this one obtains the following estimates denoting $%
z_{0}=x$
\begin{eqnarray*}
\mathbb{P}_{x}\left( \tau _{y,r}<n\right) &\geq &\mathbb{P}_{x}\left(
\mathbb{A}_{l}\right) \\
&=&\sum_{z_{l-1}\in \partial B_{l-1}}\mathbb{P}_{x}\left[ \mathbb{A}%
_{l-2}\cap D_{l-1}\left( z_{l-1}\right) \cap A_{l}\right]
\end{eqnarray*}%
Now we use the Markov property.%
\begin{eqnarray*}
&&\sum_{z_{l-1}\in \partial B_{l-1}}\mathbb{P}_{x}\left[ \mathbb{A}%
_{l-2}\cap D_{l-1}\left( z_{l-1}\right) \cap A_{l}\right] \\
&=&\sum_{z_{l-1}\in \partial B_{l-1}}\mathbb{P}_{x}\left[ A_{l}|\mathbb{A}%
_{l-2}\cap D_{l-1}\left( z_{l-1}\right) \right] \mathbb{P}_{x}\left[ \mathbb{%
A}_{l-2}\cap D_{l-1}\left( z_{l-1}\right) \right] \\
&=&\sum_{z_{l-1}\in \partial B_{l-1}}\mathbb{P}_{z_{l-1}}\left(
s_{l}<m\right) \mathbb{P}_{x}\left( \mathbb{A}_{l-2}\cap D_{l-1}\left(
z_{l-1}\right) \right) \\
&\geq &\min\limits_{w\in \pi _{x,y},2r-3\leq d\left( z,w\right) \leq 4r}%
\mathbb{P}_{z}\left( \tau _{w,r}<m\right) \mathbb{P}\left( \mathbb{A}%
_{l-1}\right) ,
\end{eqnarray*}%
Denoting $q=\min\limits_{w\in \pi _{x,y},2r-3\leq d\left( z,w\right) \leq 4r}%
\mathbb{P}_{z}\left( \tau _{w,r}<m\right) $ we have%
\begin{equation*}
\mathbb{P}\left( \mathbb{A}_{l}\right) \geq q\mathbb{P}\left( \mathbb{A}%
_{l-1}\right)
\end{equation*}%
then iterating this expression gives the result.
\end{proof}

Now we can prove the main ingredient of this section, which helps to control
the probability of hitting a nearby ball.

\begin{proposition}
\label{p3}Assume that the weighted graph $\left( \Gamma ,\mu \right) $
satisfies $\left( p_{0}\right) $ and the elliptic Harnack inequality $\left(
H\right) $. Then there are $c,C,C^{\prime }>0$ such that for all $x,y\in
\Gamma ,r\geq 1,$ $n>d\left( x,y\right) -r,d\left( x,y\right) \leq 4r$%
\begin{equation*}
\mathbb{P}_{x}\left( \tau _{y,r}<n\right) \geq c\exp \left[ -C^{\prime
}l_{C}\left( x,y,n,d\left( x,y\right) -r\right) \right] .
\end{equation*}
\end{proposition}

\begin{proof}
\ If\textbf{\ }$n>\frac{2}{c_{1}}E\left( x,9R\right) $, then the statement
follows from Proposition \ref{p2}. Also if $r\leq 9$, then $\frac{R}{3r}\leq
l\leq R,$ so from $\left( p_{0}\right) $ the trivial lower estimate
\begin{equation*}
\mathbb{P}_{x}\left( \tau _{y,r}<n\right) \geq c\exp \left( -27\left( \log
\frac{1}{p_{0}}\right) l\right)
\end{equation*}%
gives the statement. If \textbf{\ }$n<\frac{2}{c_{1}}E\left( x,9R\right) $
and $r\geq 10$, then $l_{9}\left( x,y,n,R\right) >1$ and $R=\left(
3l-2\right) r-u\geq 34.$ Let us use Proposition \ref{p2} and Lemma \ref%
{lchain}. The latter one states that
\begin{equation}
\mathbb{P}_{x}\left( \tau _{y,r}<n\right) \geq \min\limits_{w\in \pi
_{x,y},2r-3\leq d\left( z,w\right) \leq 4r}\mathbb{P}_{z}^{l}\left( \tau
_{w,r}<m\right) .  \label{prod}
\end{equation}%
Consider the following straightforward estimates for $r\geq 10,R\geq 10$.
\begin{eqnarray*}
9r &\leq &10\left( r-1\right) \leq 10\left( \frac{R+u}{3l-2}-1\right) \leq
10\left( \frac{R+3l}{3l-2}-1\right) =10\frac{R+2}{3l-2} \\
&\leq &\frac{4R}{\left( l-1\right) }\leq 8\frac{R}{l}<9\frac{R}{l}.
\end{eqnarray*}%
Let us also note $r=\frac{R+u}{3l-2}>\frac{R}{4}$ for all $l>1.$ If $%
l=l_{9}\left( x,y,n,R\right) ,$
\begin{equation*}
m=\frac{n}{l}>\frac{2}{c_{1}}E\left( w,9r\right) =\frac{2}{c_{1}}E\left( w,9%
\frac{R+u}{3l-2}\right) ,
\end{equation*}%
and $2r\leq d\left( z,w\right) \leq 4r$ then we can apply Proposition \ref%
{p2} to obtain the uniform lower estimate%
\begin{equation*}
\mathbb{P}_{z}^{l}\left( \tau _{w,r}<m\right) >c
\end{equation*}%
for $w\in \pi _{x,y}$. This yields the uniform lower bound for all
probabilities in $\left( \ref{prod}\right) $.
\end{proof}

\begin{proof}[Proof of Theorem \protect\ref{tptf}]
The upper estimate of Theorem \ref{tptf} can be seen along the lines of the
proof of Theorem 5.1 in \cite{Td}. The lower bound is immediate from
Proposition \ref{p3} by using that%
\begin{equation*}
\mathbb{P}_{x}\left( T_{x,R}<n\right) \geq \mathbb{P}_{x}\left( \tau
_{y,r}<n\right) .
\end{equation*}%
and minimizing $l_{9}\left( x,y,n\right) $ for $d=d\left( x,y\right)
=2R,y\in S\left( x,2R\right) ,\frac{d}{4}=R/2\leq r<R$.
\end{proof}

\section{Very strongly recurrent graphs}

\label{sVSR}

\begin{definition}
Following \cite{B1} we say that a graph is very strongly recurrent $\left(
VSR\right) $ if there is a $c>0$ such that for all $x\in \Gamma ,r>0,w\in
\partial B\left( x,r\right) $%
\begin{equation*}
\mathbb{P}_{w}\left( \tau _{x}<T_{x,2r}\right) \geq c.
\end{equation*}
\end{definition}

In this section we deduce an off-diagonal heat kernel lower bound for very
strongly recurrent graphs. The proof is based on Theorem \ref{tptf} and the
fact that very strong recurrence implies the elliptic Harnack inequality
(c.f. \cite{B1}). Let us mention here that the strong recurrence was defined
among others in \cite{Td} and one can see easily that strong recurrence in
conjunction with the elliptic Harnack inequality is equivalent to very
strong recurrence. It is worth to note, that the usually considered finitely
ramified fractals and their pre-fractal graphs are (very) strongly
recurrent. \

\begin{theorem}
\label{tsGE}Let us assume that $\left( \Gamma ,w\right) $ satisfies $\left(
p_{0}\right) $ and is very strongly recurrent furthermore satisfies $\left(
\overline{E}\right) $. Then there are $c,C>0$ such that for all $x,y\in
\Gamma ,n\geq d\left( x,y\right) $
\begin{equation*}
\widetilde{p}_{n}\left( x,y\right) \geq \frac{c}{V\left( x,e\left(
x,n\right) \right) }\exp \left[ -Cl_{9}\left( x,y,\frac{1}{2}n,d\right) %
\right] ,
\end{equation*}%
where $d=d\left( x,y\right) .$
\end{theorem}

\begin{remark}
Typical examples for very strongly recurrent graphs are pre-fractal
skeletons of p.c.f. self similar sets (for the definition, and further
reading see \cite{B1} and \cite{B2}). \ We recall an example of a very
strongly recurrent graph for which volume doubling does not hold but the
elliptic Harnack inequality does. The example is due to Barlow (Lemma
5.1,5.2 of \cite{B1}) and Delmotte's (c.f. \cite{D} Section 5.). \ Let us
consider $\Gamma _{1},\Gamma _{2}$ two trees which are $\left( VSR\right) $
and assume that $V_{i}\left( x,R\right) \simeq R^{\alpha _{i}},E\left(
x,R\right) \simeq R^{\beta _{i}},$ $\alpha _{1}\neq \alpha _{2},$%
\begin{equation*}
\gamma =\beta _{1}-\alpha _{1}=\beta _{2}-\alpha _{2}>0
\end{equation*}%
which basically means that
\begin{equation*}
\rho \left( x,R,2R\right) \simeq R^{\gamma }
\end{equation*}%
for both graphs. Such trees are constructed in \cite{B1}.\ Let $\Gamma $ be
the joint of $\Gamma _{1}$\ and $\Gamma _{2},$ which means that two vertices
$O_{1},O_{2}$ \ are chosen and identified (for details see \cite{D}). \ One
can also see that $\Gamma $ is $\left( VSR\right) $ and hence satisfies the
Harnack inequality but not the volume doubling property. \ This means that $%
\Gamma $ is an example for graphs that satisfies the Harnack inequality\ but
not the usual volume properties.
\end{remark}

It was realized some time ago that the so-called near diagonal lower
estimate $\left( \ref{NDLE}\right) $ is a crucial step to obtain
off-diagonal lower estimates. \ Here we utilize the fact that the near
diagonal lower bound is an easy consequence of very strong recurrence. As we
shall see the proof does not use the diagonal upper estimate and assumption
on the volume.

\begin{proposition}
\label{pDLE}Assume $\left( p_{0}\right) $ and $\left( \overline{E}\right) $,
then there is a $c>0$ such that for all $x\in \Gamma ,n>0$
\begin{equation*}
p_{2n}\left( x,x\right) \geq \frac{c}{V\left( x,e\left( x,2n\right) \right) }%
.
\end{equation*}
\end{proposition}

For the proof see Proposition 6.4 of \cite{Td}.

\begin{proposition}
$\label{pNDLE}$Let us assume that $\left( \Gamma ,\mu \right) $ satisfies $%
\left( p_{0}\right) $. If the graph is very strongly recurrent and $\left(
\overline{E}\right) $ holds, then there are $c,c^{\prime }>0$ such that for
all $x,y\in \Gamma ,m\geq \frac{2}{c^{\prime }}E\left( x,2d\left( x,y\right)
\right) $%
\begin{equation}
\widetilde{p}_{m}\left( x,y\right) \geq \frac{c}{V\left( x,e\left(
x,m\right) \right) }.  \label{NDLE}
\end{equation}
\end{proposition}

\begin{proof}
The proof starts with a first hit decomposition and uses Proposition \ref%
{pDLE}.
\begin{eqnarray*}
\widetilde{p}_{m}\left( y,x\right) &\geq &\sum_{i=0}^{m-1}\mathbb{P}%
_{y}\left( \tau _{x}=i\right) \widetilde{p}_{m-i}\left( x,x\right) \geq
\mathbb{P}_{y}\left( \tau _{x}<m\right) \widetilde{p}_{m}\left( x,x\right) \\
&\geq &\frac{c}{V\left( x,e\left( x,m\right) \right) }\mathbb{P}_{y}\left(
\tau _{x}<m\right) .
\end{eqnarray*}
Denote $r=d\left( x,y\right) ,$%
\begin{equation*}
\mathbb{P}_{y}\left( \tau _{x}<m\right) \geq \mathbb{P}_{y}\left( \tau
_{x}<T_{x,2r}<m\right) \geq \mathbb{P}_{y}\left( \tau _{x}<T_{x,2r}\right) -%
\mathbb{P}_{y}\left( T_{x,2r}\geq m\right) .
\end{equation*}%
From $\left( VSR\right) $ \ we have that $\mathbb{P}_{y}\left( \tau
_{x}<T_{x,2r}\right) >c$ so from $m\geq \frac{2}{c^{\prime }}E\left(
x,2r\right) $ and from the Markov inequality it follows that
\begin{equation*}
\mathbb{P}_{y}\left( T_{x,2r}\geq m\right) \leq \frac{E\left( x,2r\right) }{m%
}\leq c^{\prime }/2.
\end{equation*}%
Consequently we have that $\mathbb{P}_{y}\left( \tau _{x}<m\right)
>c^{\prime }/2$ and the result follows.
\end{proof}

\begin{proof}[Proof of Theorem \protect\ref{tsGE}\protect\bigskip ]
If $l=l_{9}\left( x,y,n,d\left( x,y\right) \right) =1$ , then $n>\frac{2}{%
c^{\prime }}E\left( x,9d\right) >\frac{2}{c^{\prime }}E\left( x,2d\right) $
and the statement follows from Proposition $\text{\ref{pNDLE}. }$Let us
assume that $l>1$ and start with a path decomposition. Denote $%
m=\left\lfloor \frac{n}{l}\right\rfloor ,$ $r=\left\lfloor \frac{R}{l}%
\right\rfloor ,$ $S=\left\{ y:d\left( x,y\right) =r\right\} ,$ $\tau =\tau
_{S}$
\begin{eqnarray*}
\widetilde{p}_{n}\left( y,x\right) &=&\frac{1}{\mu \left( x\right) }\mathbb{P%
}_{y}\left( X_{n}=x\text{ or }X_{n+1}=x\text{ }\right) \\
&\geq &\sum_{i=0}^{n-m-1}\sum_{w\in S}\mathbb{P}_{y}\left( X_{\tau }=w,\tau
=i\right) \min\limits_{w\in S}\widetilde{p}_{n-i}\left( w,x\right) \\
&\geq &\sum_{i=0}^{n-m-1}\mathbb{P}_{y}\left( \tau =i\right)
\min\limits_{w\in S}\widetilde{p}_{n-i}\left( w,x\right) .
\end{eqnarray*}%
The next step is to use the near diagonal lower estimate:
\begin{eqnarray*}
\widetilde{p}_{n}\left( y,x\right) &\geq &\sum_{i=0}^{n-m-1}\mathbb{P}%
_{y}\left( \tau =i\right) \min\limits_{w\in S}\widetilde{p}_{n-i}\left(
w,x\right) \\
&\geq &\sum_{i=0}^{n-m-1}\mathbb{P}_{y}\left( \tau =i\right) \frac{c}{%
V\left( x,e\left( x,n-i\right) \right) } \\
&\geq &\mathbb{P}_{y}\left( \tau <\frac{n}{2}\right) \frac{c}{V\left(
x,e\left( x,n\right) \right) }.
\end{eqnarray*}%
In the proof of Theorem \ref{tptf} we have seen that
\begin{equation*}
\mathbb{P}_{y}\left( \tau _{x,r}<\frac{n}{2}\right) \geq c\exp -Cl_{9}\left(
x,y,\frac{n}{2},d-r\right) ,
\end{equation*}%
which finally yields that
\begin{eqnarray*}
\widetilde{p}_{n}\left( y,x\right) &\geq &\frac{c}{V\left( x,e\left(
x,n\right) \right) }\exp -Cl_{9}\left( x,y,\frac{n}{2},d-r\right) \\
&\geq &\frac{c}{V\left( x,e\left( x,n\right) \right) }\exp -Cl_{9}\left( x,y,%
\frac{1}{2}n,d\right) .
\end{eqnarray*}
\end{proof}

\section{Heat kernel lower bound for graphs}

\label{sle}In this section the following off-diagonal lower bound is proved.
\

\begin{theorem}
\label{tLE}Let us assume that the graph $\left( \Gamma ,\mu \right) $
satisfies $\left( p_{0}\right) $. We also suppose that $\left( \overline{E}%
\right) $ and the elliptic Harnack inequality $\left( H\right) $ hold. Then
there are $c,C,D>0$ constants such that for all $x,y\in \Gamma ,n\geq
d\left( x,y\right) $%
\begin{equation*}
\widetilde{p}_{n}\left( x,y\right) \geq \frac{c}{V\left( x,e\left(
x,n\right) \right) r^{D}}\exp \left( -Cl_{9}\left( x,y,\frac{n}{2}\right)
\right) .
\end{equation*}%
where $e\left( x,n\right) $\ is the inverse of $E\left( x,R\right) \ $in the
second variable and $l=l_{9}\left( x,y,\frac{n}{2}\right) ,$ $d=d\left(
x,y\right) ,$ $r=\frac{d}{3l}.$
\end{theorem}

\begin{corollary}
\label{cLE}If we assume\ in addition to the conditions of Theorem \ref{tLE}
that $\beta ^{\prime }>1$ in $\left( \ref{betaprime}\right) $ then the
following more readable estimate holds:%
\begin{equation*}
\widetilde{p}_{n}\left( x,y\right) \geq \frac{c}{V\left( x,e\left(
x,n\right) \right) r^{D}}\exp \left( -C\left[ \frac{E\left( x,d\right) }{n}%
\right] ^{\frac{1}{\beta ^{\prime }-1}}\right) .
\end{equation*}
\end{corollary}

This corollary is an easy consequence of Theorem \ref{tLE}.

\begin{remark}
Let us rephrase the statement of Theorem \ref{tLE} and Corollary \ref{cLE}.
Denote $l=l_{9}\left( x,y,\frac{n}{2}\right) $. The trivial recalculation of
the estimate
\begin{eqnarray*}
\widetilde{p}_{n}\left( x,y\right) &\geq &\frac{c}{V\left( x,e\left(
x,n\right) \right) r^{D}}\exp \left( -Cl\right) \\
&=&\frac{c}{V\left( x,e\left( x,n\right) \right) d\left( x,y\right) ^{D}}%
\exp \left( D\log 3l-Cl\right) \\
&\geq &\frac{c}{V\left( x,e\left( x,n\right) \right) d\left( x,y\right) ^{D}}%
\exp \left( -Cl\right)
\end{eqnarray*}%
clearly shows the difference between the classical lower bound and the
present one. If $\left( \ref{TC}\right) $ and $\left( \ref{betaprime}\right)
$ hold with $\beta ^{\prime }>1$ furthermore $n<c\frac{d^{\beta }}{\left(
\log E\left( x,d\right) \right) ^{\beta -1}}$ then the extra factor $%
d^{D}\left( x,y\right) $ is absorbed by the exponent:%
\begin{equation*}
\widetilde{p}_{n}\left( x,y\right) \geq \frac{c}{V\left( x,e\left(
x,n\right) \right) }\exp \left( -C\left[ \frac{E\left( x,d\right) }{n}\right]
^{\frac{1}{\beta ^{\prime }-1}}\right) .
\end{equation*}
\end{remark}

\begin{proposition}
\label{pmodc}Let us assume that $\left( p_{0}\right) ,\left( \overline{E}%
\right) $ and the elliptic Harnack inequality $\left( H\right) $ holds. Then
there are $D,c>0$ such that for $x,y\in \Gamma ,$ $r=d\left( x,y\right) ,$ $%
m>CE\left( x,r\right) $ the inequality
\begin{equation*}
\widetilde{p}_{m}\left( y,x\right) \geq \frac{c}{V\left( x,e\left(
x,m\right) \right) }r^{-D}
\end{equation*}%
holds.
\end{proposition}

\begin{proof}
The proof is based on a modified version of the chaining argument used in
the proof of Lemma \ref{lchain}. From Proposition \ref{pNDLE} we know that $%
\left( \overline{E}\right) $ implies
\begin{equation}
\widetilde{p}_{n}\left( x,x\right) \geq \frac{c}{V\left( x,e\left(
x,n\right) \right) }  \label{ldle}
\end{equation}%
and $\left( \ref{PD3E}\right) $ (see Remark \ref{rEbar}). Let us recall $%
\left( \ref{PD3E}\right) $ and set $A=\max \left\{ 9,A_{T}\right\} ,$ $%
K=\left\lceil \frac{A}{4}\right\rceil $. \ Consider a sequence of times $%
m_{i}=\frac{m}{2^{i}}$ and radii $r_{i}=\frac{r}{A^{i}}.$ \ From the
condition $m>CE\left( x,r\right) $ and $\left( \ref{PD3E}\right) $ it
follows that for all $i$%
\begin{equation}
m_{i}>CE\left( x,r_{i}\right)  \label{m>e}
\end{equation}%
holds as well. \ Let us denote $B_{i}=B\left( x,r_{i}\right) $,$\tau
_{i}=\tau _{B_{i}}$ and start a chaining.
\begin{eqnarray*}
\widetilde{p}_{m}\left( y,x\right) &=&\sum_{k=1}^{m}\mathbb{P}_{y}\left(
\tau _{1}=k\right) \min\limits_{w\in \partial B_{1}}\widetilde{p}%
_{m-k}\left( w,x\right) \\
&\geq &\sum_{i=1}^{m/2}\mathbb{P}_{y}\left( \tau _{1}=k\right)
\min\limits_{w\in \partial B_{1}}\widetilde{p}_{m-k}\left( w,x\right) \\
&\geq &\mathbb{P}_{y}\left( \tau _{1}<m/2\right) \min\limits_{1\leq k\leq
m/2}\min\limits_{w\in \partial B_{1}}\widetilde{p}_{m-k}\left( w,x\right) .
\end{eqnarray*}%
Let us continue in the same way for all $i\leq L:=\left\lceil \log
_{A}r\right\rceil $. 

It is clear that $B_{L}=\left\{ x\right\} $ which concludes to%
\begin{eqnarray*}
\widetilde{p}_{m}\left( y,x\right) &\geq &\min\limits_{w_{i}\in \partial
B_{i}}P_{y}\left( \tau _{1}<m/2\right) ... \\
&&...P_{w_{j}}\left( \tau _{j}<\frac{m}{2^{i}}\right) ..P_{w_{L}}\left( \tau
_{L}<\frac{m}{2^{L}}\right) \min\limits_{0\leq k\leq m-L}\widetilde{p}%
_{k}\left( x,x\right) .
\end{eqnarray*}%
\ From the initial conditions and $\left( \ref{PD3E}\right) $ we have $%
\left( \ref{m>e}\right) $ for all $\ j$.

Since in the consecutive steps $d\left( w_{i},x\right) >4r_{i+1}$ we insert $%
K-1$ copies of balls of radius $r_{i+1}$ splitting the distance into equal
smaller ones. \ We do chaining along them prescribing that the consecutive
balls are reached in less than $m_{i}/K$ time. \ We can choose $C$ so that
the conditions of Proposition \ref{p2} are satisfied which yields
\begin{equation*}
\mathbb{P}_{w}\left( \tau _{i}<\frac{m}{2^{i}}\right) >c_{0}^{K}
\end{equation*}%
for all $w_{j}\in B\left( x,r_{j}\right) $ and $j.$ Consequently, using $%
\left( \ref{ldle}\right) $ one has
\begin{equation*}
\widetilde{p}_{m}\left( y,x\right) \geq \frac{c}{V\left( x,e\left(
x,m\right) \right) }c_{2}^{L}\geq \frac{c}{V\left( x,e\left( x,m\right)
\right) }r^{-D}
\end{equation*}%
where $D=$ $\frac{\log \frac{1}{c_{2}}}{\log A}.$
\end{proof}

\begin{proof}[Proof of Theorem \protect\ref{tLE}]
The proof is a combination of two chaining arguments. \ First let us use
Theorem $\ref{tptf}$ to reach the boundary of $B\left( x,r\right) ,$ where $%
r=\frac{d\left( x,y\right) }{3l-1},l=l_{9}\left( x,y,\frac{n}{2},d\left(
x,y\right) \right) $, then we use Proposition \ref{pmodc}. \
\end{proof}

\end{document}